\documentclass[leqno,12pt]{amsart}

\usepackage{hyperref}
\usepackage{enumitem}
\usepackage[active]{srcltx}

\usepackage{amsmath}
\usepackage{amssymb}
\usepackage{amscd}
\usepackage{amsthm}
\usepackage{cleveref}
\usepackage{pdfsync}
\usepackage{pictexwd,dcpic}
\usepackage[usenames,dvipsnames,svgnames,table]{xcolor}
\usepackage{tikz}
\usepackage{tikz-cd}

\newcommand{\nons}[1]{}

\usepackage{mathrsfs}
\usepackage[latin1]{inputenc}
\usepackage{graphics}

\newtheorem{teo}{Theorem}[section]

\newtheorem{defin}[teo]{Definition}
\newtheorem{prop}[teo]{Proposition}
\newtheorem{cor}[teo]{Corollary}
\newtheorem{lemma}[teo]{Lemma}

\newtheorem{quest}[teo]{Question}
\theoremstyle{definition}
\newtheorem{remark}[teo]{Remark}

\newtheoremstyle{dico}
 {\baselineskip}   
  {\topsep}   
  {}  
  {0pt}       
  {} 
  {.}         
  {5pt plus 1pt minus 1pt} 
  {}          
\theoremstyle{dico}

\numberwithin{equation}{section}

\newcommand{\Ga}{\Gamma}

\newcommand{\meno}{^{-1}}

\newcommand{\Zeta}{{\mathbb{Z}}}

\newcommand{\PP}{\mathbb{P}}

\newcommand{\proba}{\mathscr{P}}
\newcommand{\pb}{\proba(X)}
\newcommand{\misu}{\mathscr{M}}

\newcommand{\nuo}{{\nu_0}}

\newcommand{\convo}{E( \mu)}

\newcommand{\intec}{\Omega(\mu)}
\newcommand{\into} {\intec}

\newcommand{\liek}{\mathfrak{k}}

\newcommand{\supp}{\operatorname{supp}}
\newcommand{\spam}{\,\operatorname{span}\, }

\newcommand{\alfa}{\alpha}

\newcommand{\vacuo}{\varnothing}

\newcommand{\enf}{\emph}

\newcommand{\restr}[1]          {\vert_{#1}}

\newcommand{\sx}{\langle}
\newcommand{\xs}{\rangle}
\newcommand{\scalo}{\sx \cdot , \cdot \xs}

\newcommand{\Aut}{\operatorname{Aut}}

\newcommand{\Sl}{\operatorname{SL}}
\newcommand{\Gl}{\operatorname{GL}}

\newcommand{\PGL}{\operatorname{\PP GL}}

\newcommand{\End}{\operatorname{End}}
\newcommand{\lds}{\ldots}
\newcommand{\cds}{\cdots}
\newcommand{\cd}{\cdot}
\renewcommand{\setminus}{-}

\newcommand{\ra}{\rightarrow}
\newcommand{\lra}{\longrightarrow}
\newcommand{\C}{\mathbb{C}}

\newcommand{\de}{\partial}

\newcommand{\om}{\omega}

\newcommand{\eps}{\varepsilon}
\renewcommand{\phi}{\varphi}

\renewcommand{\Bigl}{\left}

\renewcommand{\Bigr}{\right}

\newcommand{\fun}{\mathfrak{F}}
\newcommand{\bly}{F}

\newcommand{\id}{\operatorname{id}}
\newcommand{\y}{Y}
\newcommand{\z}{Z}
\newcommand{\M}{X}

\newcommand{\ratto}{\dashrightarrow}
\newcommand{\mero}{\dashrightarrow}
\newcommand{\Mero}{\mathcal{M}}
\newcommand{\pr}{\pi} 
\newcommand{\mult}{\operatorname{mult}}

  \newcommand{\weak}{\rightharpoonup}
\newcommand{\fuj}{{\mathscr C}}

\newcommand{\au}{{{B}^0(X)}}
\newcommand{\aus}{{{B}(X)}}
\newcommand{\ai}{{\Aut^0(X)}}
\newcommand{\Autz}{\Aut^0}
\newcommand{\B}{{\partial \aus }}
\newcommand{\inde}{\operatorname{indet}}
\newcommand{\A}{\hat{A}}
\newcommand{\Alb}{\operatorname{Alb}}
\newcommand{\alb}{\operatorname{alb}}
\newcommand{\cicl}{{C}^{\operatorname{loc}}}
\newcommand{\cic}{{C}}

\newcommand{\Dou}{\mathscr{D}}
\newcommand{\Fu}{F}
\newcommand{\red}{{\operatorname{red}}}

\begin{document}

\title{Meromorphic limits of automorphisms}

\author{Leonardo Biliotti}

\author{Alessandro Ghigi}

\begin{abstract}
  Let $X$ be a compact complex manifold in the Fujiki class $\fuj$.
  We study the compactification of $\ai$ given by its closure in
  Barlet cycle space. The boundary points give rise to non-dominant
  meromorphic self-maps of $X$.  Moreover convergence in cycle space
  yields convergence of the corresponding meromorphic maps.  There are
  analogous compactifications for reductive subgroups acting trivially
  on $\operatorname{Alb} X$.  If $X$ is K\"ahler, these
  compactifications are projective.  Finally we give applications to
  the action of $\Aut(X)$ on the set of probability measures on $X$.
  In particular we obtain an extension of Furstenberg lemma to
  manifolds in the class $\fuj$.
\end{abstract}

\address{Universit\`{a} di Parma} \email{leonardo.biliotti@unipr.it}
\address{Universit\`a di Pavia}\email{alessandro.ghigi@unipv.it}

\thanks{Both authors were partially supported by FIRB 2012
    ``Geometria differenziale e teoria geometrica delle funzioni'' and
    by INdAM - GNSAGA.  The first author was also supported by MIUR
    PRIN 2015 ``Real and complex manifolds: geometry, topology and
    harmonic analysis''.  The second author was also supported by MIUR
    PRIN 2015 ``Moduli spaces and Lie Theory'' , by MIUR FFABR, by FAR
    2016 (Pavia) ``Variet\`a algebriche, calcolo algebrico, grafi
    orientati e topologici'' and by MIUR, Programma Dipartimenti di
    Eccellenza (2018-2022) - Dipartimento di Matematica
    ``F. Casorati'', Universit\`a degli Studi di Pavia.}



\maketitle

\tableofcontents

\section{Introduction}

Let $X$ be a compact complex manifold and assume that $\ai$, the
connected component of $\Aut(X)$ containing the identity, is not
trivial.  It is interesting to consider pointwise limits of sequences
$\{g_n\} $ in $\ai$.  Even more interesting is the fact that such
limits often exist!  We first met with this phenomenon in the case of
a rational homogeneous space $X=G/P$.  Fix an ample class on $X$ and a
Cartan involution $\theta$ on $G$. Call \emph{self-adjoint} the
elements $g\in G$ such that $\theta(g) = g\meno$. These elements form
a submanifold of $G$ diffeomorphic to the symmetric space $G/K$, where
$K=\mathrm{Fix}(\theta)$.  The ample class allows to fix a particular
Satake compactification of $G/K$.  One can prove that if a sequence
$\{g_n\}$ of self-adjoint elements converges in the Satake
compactification, then the maps $g_n : X \ra X$ converge almost
everywhere on $X$ (with respect to smooth Lebesgue measures).  The
limit map is a rational self-map of $X$ and one can describe it rather
explicitely, see \cite[\S 3.1]{bgAmer}.  In particular the pointwise
limit of the maps $g_n$ exists, it is holomorphic on a Zariski open
subset of $X$ and its image is contained in a proper subvariety of
$X$.

We later discovered that this phenomenon holds in greater generality.
Assume that $X$ is a K\"ahler manifold and that a compact connected
subgroup $K \subset \ai$ acts on $X$ in a Hamiltonian way, i.e. with a
momentum mapping.  If $\xi \in \liek$ and $x\in X$, then the limit
\begin{gather}
  \label{eq:2}
  \lim_{t\to+\infty} \exp (it\xi) \cd x
\end{gather}
always exists and defines a limit map, see
e.g. \cite[Prop. 5.18]{bgs}. This map is not continuous on the whole
manifold $X$, but its restriction to a Zariski open subset is
continuous and holomorphic \cite[\S 5.20]{bgs}.  If we set
$g_n (x) : = \exp (it_n \xi)\cd x$ for a sequence $\{t_n\}$ converging
to $ +\infty$, then we observe the same phenomenon as above: the
pointwise limit of $g_n$ exists and is holomorphic on a Zariski open
subset of $X$.  The proof of these facts relies heavily on the
Linearization Theorem proved in the papers \cite
{heinzner-huckleberry-Inventiones}, \cite{heinzner-loose}, \cite[\S
14]{heinzner-schwarz-Cartan}. As is well-known the flow $\exp(it \xi)$
in \eqref{eq:2} is a Morse-Bott flow.  It is interesting to notice
that using quite different methods one can make sense of the limit for
every Morse-Bott flow, see \cite{hl,latschev}.

In the present paper we study this phenomenon, that is the existence
of the limit, in full generality:
\begin{quest}
  \label{que}
  Let $X$ be a compact complex manifold and let $\{g_n\}$ be a
  sequence in $\ai$. For which $x\in X$ does the limit
  \begin{gather*}
    f(x):=\lim_{n\to \infty}  g_n \cd x
  \end{gather*}
  exist (up to passing to a subsequence)?  What is the structure of
  the set of such points?  What can be said about the limit map $f$?
\end{quest}

The basic idea of our approach is simply to replace a biholomorphism
of $X$ by its graph. This idea goes back at least to Douady
\cite{douady-modules} and is of course common in many areas of
mathematics. The graph of a biholomorphism is an analytic subvariety
of $X\times X$. Subvarieties can be considered either as ideal
sheaves, i.e. points in the Douady space (the Hilbert scheme in the
projective case), or as cycles, i.e. points in the Barlet cycle space
(the Chow scheme in the projective case).  For our purposes the choice
between these two approaches is not fundamental.

The manifolds for which we can answer the question above are those in
Fujiki class $\fuj$: this class contains by definition all the
manifolds that are meromorphic images of compact K\"ahler manifolds
(see Definition \ref{defclassc} below).  For these manifolds the
irreducible components of both Douady and cycle space are compact.
Let $\aus$ (respectively $\Fu(X)$) denote the irreducible component of
the diagonal in the cycle space $C_n(X\times X)$, where $n=\dim X$
(resp. in the Douady space of $X\times X$).  Thus $\aus$
(resp. $\Fu(X)$) is an analytic compactification of $\ai$.  Some
instances of this compactification have already been considered in the
literature. For example Brion \cite{brion-completions} has studied
$\aus$ in great detail in the case where $X$ is a rational homogeneous
space.  Using the compactness of $\aus$ we prove the following result,
which gives a rather complete answer to Question \ref{que} for $X$ in
the class $\fuj$ (see \S \ref{autx0}, especially Theorems
  \ref{nondominant} and \ref{convergenza}).
\begin{teo}
  \label{main1}
  Let $X \in \fuj$ and let $\{g_n\}$ be a divergent sequence in $\ai$.
  Up to passing to a subsequence there are a meromorphic map
  $f: X\mero X$ and a proper analytic subset $A\subset X$ such that
  \begin{enumerate}
  \item $f$ is defined outside $A$;
  \item $g_n \to f$ uniformly on compact subsets of $X-A$;
  \item \label{terzo} $f$ is not dominant, i.e. $f(X)$ is contained in
    a proper subvariety of $X$.
  \end{enumerate}
\end{teo}

An example of complex manifold not in the class $\fuj$ is provided by
Hopf manifolds \cite{ueno-class-C}. We are able to show that for such
manifolds our result fails, see Remark \ref{ciuccia}.

In \S \ref{ridut} we consider reductive subgroups of $\ai$. We recall
several results from Fujiki's fundamental paper
\cite{fujiki-automorphisms}.  Fujiki used $\Fu(X)$ instead of $\aus$.
We explain that they are equivalent for our purposes.  It follows that
for every connected complex reductive subgroup $G\subset \ai$ that
acts trivially on $\Alb X$, the closure $\overline{G} \subset \aus$ is
analytic. (The corresponding statement in $\Fu(X)$ was proved by
Fujiki.)  This allows to refine (\ref{terzo}) in Theorem \ref{main1}:
if the sequence $\{g_n\}$ lies in $ G$, then $f (X)$ is contained in
the fixed set of a positive-dimensional subgroup of $G$.

The compactification of a reductive $G \subset \ai$ obtained in this
way is quite interesting in its own. If $X$ is K\"ahler we are able to
prove the following (see Theorem \ref{proj}).
\begin{teo}
  \label{main2}
  If $X$ is a K\"ahler manifold and $G \subset \ai$ is a connected
  complex reductive subgroup, that acts trivially on $\Alb X$, then
  the closure of ${G}$ inside $\aus$ is a projective variety.
\end{teo}

In \S \ref{misure} we apply Theorem \ref{main1} to study the action of
$\ai$ on the set of probability measures on $X$.  A famous lemma due
to Furstenberg \cite{furstenberg-Borel}, which is used in the proof of
Borel density theorem, says (among other things) that a measure on
$\PP^n$ whose stabilizer in $\operatorname{PGL}(n+1,\C)$ is
non-compact, is supported on a union of proper linear subspaces.  The
previous results allow to generalize this to any manifold in $\fuj$: a
measure on $X$ with non-compact stabilizer in $\ai$ is supported on a
proper analytic subset (see Theorem \ref{fursty-analitico}).

Finally in Theorem \ref {propria} we give an application of the
results obtained in the paper to the map $F_\nu$, originally
introduced by Bourguignon, Li and Yau \cite{bly} and studied in
\cite{bgAmer,bgs}.  We are able to give a much shorter proof of one of
the main results in \cite{bgs}, although in a slightly less general
setting.

{\bfseries \noindent{Acknowledgements}}.  The authors would like to
thank Professor Barlet for helping with cycle space, Professor Pirola
for interesting discussions and Professor Dolgachev for turning their
attention to the important paper \cite{neretin}.

\section{Notation and  preliminaries}

We start by recalling the basic definitions on meromorphic maps and
some elementary lemmata needed in the paper. See
\cite{barlet-magnusson-vol-1,fritzsche-grauert,grauert-remmert-cas,noguchi-ochiai}
for more details.

\begin{defin}
  Let $X$ and $Y$ be reduced complex spaces. A map $\tau : X \ra Y$ is
  a \enf{proper modification} if it is proper and there is an analytic
  subset $T \subset Y$ with empty interior such that
  \begin{enumerate}
  \item $\tau\meno(T)$ has empty interior and
  \item the restriction of $\tau$ to $X -\tau\meno(T) $ is a
    biholomorphism onto $Y-T$.
  \end{enumerate}
  The \enf{center} of $\tau$ is the intersection of all the analytic
  subset $T\subset Y$ satisfying the above condition. The
  \enf{exceptional set} of $\tau$ is the inverse image of the
  center. 
\end{defin}

\begin{defin}
  \label{merodef}
  Let $X$ and $Y$ be reduced complex spaces. A \enf{meromorphic map}
  of $X$ in $Y$ is an analytic subset $G$ of $X\times Y$ such that
  $p:=\pr_1\restr{G} : G \ra X$ is a proper modification.  If
  $S \subset X$ is the center of $p$ then
  $f:=\pr_2 \circ p\meno : X-S \ra Y$ is a holomorphic map. We write
  $f : X \ratto Y$. The set $G$ is called the \emph{graph} of $f$
  and it is denoted by $\Ga_f$.  The \enf{image} of $f$ is
  $\pr_2(G) \subset Y$.  The meromorphic map $f$ is \enf{surjective}
  if $\pr_2(G)=Y$.  The center of $p$ is called the \enf{set of
    indeterminacy} of $f$, denoted $\inde (f)$, and its complement is
  called the \enf{domain of definition} of $f$. We say that $f$ is
  defined at $x \in X$ if $x$ lies in the domain of definition.
\end{defin}

\begin{remark}
  If $\tau : X \ra Y$ is a proper modification and $Y$ is irreducible,
  then also $X$ is irreducible.  In fact $Y -T$ is irreducible and so
  is $X -\tau\meno(T)$. Moreover $X - \tau\meno(T)$ is dense in $X$.
  As a corollary, if $f: X \ratto Y$ is a meromorphic map with graph
  $G$, and $X$ irreducible, then $ G$ is irreducible.
\end{remark}

\begin{lemma}
\label{lemma-1}
Let $X$ and $Y$ be reduced and irreducible compact analytic spaces.
Let $f : X \ratto Y$ be a meromorphic map with graph $G$ and set of
indeterminacy $S \subset X$.  Then $G$ is the closure of the graph of
$f: X-S \ra Y$.
\end{lemma}
\begin{proof}
  Since $f: X-S \ra Y$ is a holomorphic map, its graph $\Ga_f$ is an
  analytic subset of $(X-S)\times Y$ and it is biholomorphic to $X-S$.
  By the definition of meromorphic map we have
  $\Ga_f = G - (S \times Y)$.  Therefore $\Ga_f$ is Zariski open in
  $G$. By the previous remark $G$ is irreducible, so $\Ga_f$ is dense
  in $G$ for the Hausdorff topology.
\end{proof}

\begin{lemma}
  \label{meromorfason}
  If $X$ and $Y$ are reduced and irreducible compact analytic spaces
  and $S\subset X$ is a proper analytic subset, a holomoprhic map
  $f: X -S \ra Y$ is meromorphic if and only if the closure of its
  graph is an analytic subset of $X\times Y$.
\end{lemma}
\begin{proof}
  We already proved that the condition is necessary. To prove that it
  is sufficient, assume that $G:=\overline{\Ga_f}$ is analytic in
  $X\times Y$.  Since $G$ is compact the map $p: = \pi_1 \restr{G}$ is
  proper.  Moreover $\pr_1(G) = X$, since $\pr_1(G)$ is compact and
  contains $X-S$.  Since $X$ is irreducible, also $\Ga_f$ and $G$ are
  irreducible.  Finally $p\meno(S) = G\cap (S\times X)$ is a proper
  analytic subset of $G$, so it is nowhere dense. We have proved that
  $p: G \ra X$ is a proper modification.
\end{proof}

\begin{lemma}
  \label{chiusura}
  Let $X$ and $Y$ be reduced and irreducible compact analytic spaces
  and let $f : X \mero Y$ be a meromorphic map. Let $A\subset X$ be a
  proper analytic subset containing $\inde(f)$.  If $W \subset X$ is
  an irreducible analytic subset which is not contained in $A$, then
  $f(W - A)$ has analytic closure in $Y$.
\end{lemma}
\begin{proof}
  Let $G \subset X\times Y$ be the graph of $f$ and let $\pi_1, \pi_2$
  be the restrictions of the projections:
  \begin{equation*}
    \begin{tikzcd}
      & G \arrow{dl}[swap]{\pi_1}  \arrow{dr}{\pi_2}   &   \\
      X \arrow[dashed]{rr}{f}& & Y.
    \end{tikzcd}
  \end{equation*}
  Let $\pi_1\meno (W) = Z_1 \cup \cds \cup Z_r$ be the decomposition
  in irreducible components.  Since $W$ is irreducible, we can assume
  $\pi_1(Z_1) = W$.  We claim that $\pi_2 (Z_1) = \overline{f(W-A)}$.
  Indeed since $W$ is irreducible, $W-A$ is also irreducible. Since
  $\pi_1$ is a biholomorphism over $X-A$, also
  $\pi_1 \meno(W-A) \cong W-A$ is irreducible.  Hence it is contained
  in a unique irreducible component of $\pi_1\meno(W)$, which is
  necessarily $Z_1$.  This shows that
  $\pi_1\meno(W-A) \subset Z_1 - \pi_1\meno(A)$.  The opposite
  inequality being obvious, we get
  $\pi_1\meno(W-A) = Z_1 - \pi_1\meno(A)$.  Since $Z_1$ is
  irreducible, $\pi_1\meno(W-A) $ is dense in $ Z_1$.  So
  $f (W-A) = \pi_2 \pi_1\meno(W-A) = \pi_2 ( Z_1 - \pi_1\meno(A))$ is
  dense in $\pi_2(Z)$. This means that the closure of $f(W-A)$ is the
  set $\pi_2(Z_1)$, which is analytic by Remmert Proper Mapping
  Theorem.
\end{proof}

\begin{lemma}
  \label{bimeromorfi} Let $X$ and $Y$ be reduced and irreducible
  compact analytic spaces and let $f : X \ra Y$ be a holomorphic map.
  Let $B\subset Y$ be a proper analytic subset such that for any
  $y\in Y - B$, the fibre $f\meno(y)$ consists of a single point.
  Then $f$ is a bimeromorphic map.
\end{lemma}
\begin{proof}
  Define $h: Y - B \ra X$ by $h(y):= f\meno(y)$.  Let
  $G\subset X\times Y$ denote the graph of $f$, which is an
  irreducible analytic subset of $X\times Y$.  The map
  $t: X\times Y \ra Y\times X$, $t(x,y) : = (y,x)$ is a
  biholomorphism, so also $G':=t(G)$ is analytic and irreducible in
  $Y\times X$. The set $G' - \pi_1\meno (B) $ is Zariski open in $G'$
  and it coincides with the graph of $h$.  By Lemma \ref
  {meromorfason} we conclude that $h$ extends to a meromorphic map
  $Y\mero X$. By construction we have $hf = \id_X$ on $X - f\meno(B)$
  (which is and nonempty and dense in $X$) and $fh = \id_Y$ on $Y-B$.
  Therefore $h$ is a meromorphic inverse to $f$.
\end{proof}

We will need the following classical result (see e.g. \cite[Cor. 1.20
p. 108]{peternell-differential} and \cite[p. 116]{open-demailly}).

\begin{teo}
  \label{mappotta}
  Let $X$ and $Y$ be compact complex spaces and let $f : X \ra Y$ be a
  proper surjective holomorphic map.  Assume that $X$ and $Y$ are
  reduced and irreducible.  Then there are Zariski open subsets
  $Y^0 \subset Y$ and $X^0 \subset X$ such that $f(X^0) = Y^0$, both
  $X^0$ and $Y^0$ are non-singular and $f\restr{X^0} : X^0 \ra Y^0$ is
  a submersion with fibres of dimension equal to $\dim X - \dim X$.
\end{teo}

We now recall the basic definitions related to Barlet cycle space.

\begin{defin}
  Let $\M$ be a reduced complex space. A $n$-cycle in $\M$ is a locally
  finite sum $Z=\sum_i n_i Z_i$ where $n_i \in \mathbb{N}$ and $Z_i$
  is an irreducible analytic subset of $\M$ of dimension $n$.
\end{defin}

The set of $n$-cycles in $\M$ will be denote by $\cicl_n(\M)$.
A cycle is \emph{compact} if the subsets $Z_i$ are compact and $n_i\neq 0$
for only finitely many indices.
The set of compact $n$-cycles in $\M$ will be denote by $\cic_n(\M)$.
It can
be provided with the structure of a Banach analytic space. The
irreducible components have finite dimension.  A family of
$n$-dimensional cycles in $\M$ parametrized by a topological space $S$
is a map $f:S \ra C_n(\M)$. We also denote the family by
$\{\y_s:=f(s)\}_{s\in S}$.  The family is called continuous if the
corresponding map is continuous. It is called analytic if $S$ is a
complex space and the map is holomorphic.

The universal family of $n$-cycles in $\M$ is the analytic family
corresponding to to the identity map of $C_n(\M)$
\cite[p. 367]{barlet-magnusson-vol-1}.

An $n$-cycle $\y$ on $\M$ has a well-defined multiplicity $\mult_x(\y)$
at every point $x\in \M$ \cite[p. 446]{barlet-magnusson-vol-1}.

Let $(\y_s)_{s\in S}$ be an analytic family of $n$-cycles on $\M$.  The
\enf{set-theoretic graph} of the family is the analytic subset
\begin{gather}
  |G_S|:= \{(s, x) \in S\times \M: \mult_x(\y_s) \geq 1\}.
\label{def-grafico}
\end{gather}
Let $|G_S|= \cup G_i$ be the decomposition in irreducible components.
For each $i$ the function $(s,x) \mapsto \mult_x(\y_s)$ has a generic
value $n_i$ on $G_i$.  Then $ G:= \sum_i n_i G_i $ is the \enf{graph}
of the family. It is an $n+q$-cycle on $S\times \M$, where $S$ is
reduced and has pure dimension $q$. This cycle is compact if and only
if $S$ is compact.

\begin{teo}[\protect{
\cite  [Thm. 3.3.1 p. 448]{barlet-magnusson-vol-1}}]
\label{molteplicita-componenti}
For very general $s\in S$ let $\y_s = \sum_k n_{s,k} Z_{s,k}$ be the
decomposition in irreducible components. Then $n_{s,k} = n_i$ if
$\{s\}\times Z_{s,k} \subset G_i$.
\end{teo}

\begin{teo} [\protect{\cite[Thm. 3.4.1
    p. 449]{barlet-magnusson-vol-1}}]
  \label{barlet1}
  Let $S$ be a normal complex space and let $G\in \cicl_{n+q}(S\times \M)$.
  Assume that the fibres of $\pi: |G|\ra S$ have pure dimension $n$
  and that $\pi$ is proper. Then there is a unique analytic family of
  cycles whose graph is $G$.
\end{teo}

\begin{defin}
  \label{defclassc} A complex manifold $X$ is said to belong to the
  Fujiki class $\mathscr{C}$ if there is a compact K\"ahler manifold
  $Y$ and a surjective meromorphic map $h:Y \mero X$.  By Hironaka's
  theorem one can assume that $h$ is holomorphic.  Moreover in
  \cite{varouchas-image,barlet-varouchas} it is proven that $h$ can be
  assumed to be bimeromorphic.  For more details see \cite[\S
  4.3]{fujiki-closedness},\cite{ueno-class-C,varouchas-image,barlet-varouchas}.
\end{defin}

The following result due to Campana and Fujiki is fundamental for the
whole paper.  See \cite[p. 431]{barlet-magnusson-vol-1} for a proof in
the K\"ahler case and \cite{fujiki-Douady,campana-algebricite} for the
general case.
\begin{teo}\label{C}
  If $\M$ is a reduced complex space in class $\fuj$, then any
  irreducible component of $C_n(\M)$ is compact.
\end{teo}

\section{Limit maps for sequences in $\Aut^0(X)$}
\label{autx0}

Let $X$ be an $n$-dimensional compact connected complex manifold in
the class $\fuj$.  For $f\in \ai$, let $\Ga_f \subset X\times X$
denote the graph of $f$.  Since $X$ is a connected manifold, the graph
is an irreducible analytic subset. In particular
$\Ga_f \in C_n(X\times X)$.  This yields a map
\begin{gather}
  \label{def-j}
  j : \ai \ra C_n(X\times X),\qquad j(f):=\Ga_f.
\end{gather}
We denote by $\au$ the image of $j$ and by $\aus$ the closure of $\au$
in $C_n (X\times X)$.  We will often identify $f\in \ai$ with $j(f)$
and consider $\ai$ as a subset of $\aus$.  The idea of replacing $f$
by its graph goes back to \cite{douady-modules} and has been used in
\cite{lieberman-automorphisms} and \cite{fujiki-automorphisms}.  Also
the following Proposition has been proven in
\cite{fujiki-automorphisms,lieberman-automorphisms}.
\begin{prop}
  \label{lieber} The map $j$ is a holomorphic embedding, $\aus$ is an
  irreducible component of $C_n(X\times X )$ and $\B:=\aus - \au$ is
  an analytic subset of $\aus$.
\end{prop}
\begin{proof}
  To prove that $j$ is holomorphic it is enough to prove that the
  family of cycles $(\Ga_f)_{f\in \ai}$ is analytic. Indeed
  $Z:= \{ (f, x, y) \in \ai\times X \times X: f(x) = y\}$ is a complex
  submanifold of $\ai\times X\times X$ biholomorphic to $\ai\times X$,
  hence irreducible. By Theorem \ref{barlet1} it defines an analytic
  family, which corresponds to the map $j$. The image of $j$ is
  contained in a unique irreducible component of $C_n(X\times X)$ that
  we denote by $\aus$.  The rest is proven in
  \cite[Prop. 2.1]{lieberman-automorphisms}.
\end{proof}

It follows from Theorem \ref{C} that $\aus$ is a compact irreducible
analytic space. In fact it belongs to class $\mathscr{C} $
\cite[Cor. 3]{campana-algebricite}.  The inclusion
$\aus \hookrightarrow C_n(X\times X)$ corresponds to a family of
$n$-cycles on $X\times X$ that we denote by $\{\y_b\}_{b\in \aus}$.
In other words $\{\y_b\}_{b\in \aus}$ is the restriction of the
universal family of cycles to $\aus \subset C_n(X\times X)$.  Let
$G_{\aus } $ be the graph of the family $\{\y_b\}_{b\in \aus}$.

\begin{lemma}
  \label{interz}
  For any $b\in \aus$ and any $x \in X$ the intersection
  $\y_b \cap (\{x\}\times X)$ is non-empty. It either contains a
  component of positive dimension or it reduces to a single point. In
  the latter case this point is a smooth point of $\y_b$, at which
  $Y_b$ and $\{x\}\times X$ is intersect transversally.
\end{lemma}
\begin{proof}
  Since $\aus$ is connected, the homology class of $\y_b$ is constant
  for $b\in \aus$. In particular it coincides with the homology class
  of the diagonal $\Delta$, which is the graph of the identity map of
  $X$.  Setting for simplicity $F_x:=\{ x\} \times X$, in the homology
  ring of $X\times X$ we have
  \begin{gather*}
    [\y_b] \cd [F_x] = [\Delta] \cd [F_x] .
  \end{gather*}
  Since $\Delta$ and $F_x$ intersect only at $(x,x)$ and the
  intersection is transverse, $ [\Delta] \cd [F_x] =1$ and therefore
  $ [\y_b] \cd [\{x\} \times X]=1$.  It follows immediately that
  $\y_b \cap F_x \neq \vacuo$.  This intersection is a compact
  analytic subset of $X\times X$. If there are no components of
  positive dimension, then $\y_b \cap F_x = \{p_1, \lds, p_k\}$. So
  $\y_b$ and $F_x$ intersect \emph{properly} and
  \begin{gather*}
    1 = [\y_b] \cd [F_x] = \sum_{i=1}^k I(p_i, \y_b , F_x, X\times X).
  \end{gather*}
  Since $I(p_i, \y_b , F_x, X\times X) \geq 1$, we conclude that
  $k=1$, i.e. $\y_b\cap F_x = \{p_1\}$ and also that
  $I(p_1, \y_b, F_x , X\times X) = 1$. It follows that both $\y_b$ and
  $F_x$ are smooth at $p_1$ and that they are transversal, see
  \cite[p. 137-138]{fulton}.
\end{proof}

Given spaces $X_1, X_2, \lds X_n$ we denote by $\pi_i$ and
$\pi_{i,j} $ the natural projections
\begin{gather*}
  \pi_{i} : X_1\times X_2 \times \cds \times X_n \lra X_i \\
  \pi_{ij} : X_1\times X_2 \times \cds \times X_n \lra X_i\times X_j
\end{gather*}

\begin{lemma}
  \label{cavallo}
  Assume that $X\in \fuj$. Set
  \begin{gather*}
    \psi:=\pi_{12}\restr{|G_\aus|} : |G_\aus| \lra \aus \times X,
    \\
    \psi(b, x_1, x_2 ) :=(b, x_1).
  \end{gather*}
  \begin{enumerate}[label={(\roman*)}]
  \item \label{uno} The map $\psi$ is onto.
  \item \label{due} The set
    $ \Omega := \{ (b, x) \in \aus \times X : |\psi\meno (b,x)| = 1\}
    $ is Zariski open in $\aus \times X$.
  \item \label{trebis} The restriction
    $ \psi\restr{\psi\meno (\Omega ) } : \psi\meno (\Omega ) \lra
    \Omega $
    is a homeomorphism.
  \item \label{tre} If $(b,x) \in \Omega$, then there is an open
    neighbourhood $U$ of $x$ in $X$ and a holomorphic function
    $\phi: U \ra X$ such that $\y_b \cap (U\times X)$ coincides with
    the graph of $\phi$.
  \item \label{quaccio} $\au \times X \subset \Omega$.
  \item \label{quinque} The set-theoretic graph $|G_\aus|$ is
    irreducible and $G_\aus = |G_\aus|$.
  \item \label{sei} For any $b\in B$ we have
    $ \Omega \cap ( \{b\}\times X) \neq \vacuo$.
  \item \label{septimius} If $b \in \aus $ there is one and only one
    irreducible component $Z_b$ of $\y_b$ such that $\pi_1(Z_b)=X$.
    This component has multiplicity 1 in $\y_b$ and it is the graph of
    a meromorphic map $f_b: X \mero X$.
  \end{enumerate}
\end{lemma}
\begin{proof}
  Recall that
  $|G_\aus| = \{ (b, x_1, x_2) \in \aus \times X \times X : (x_1, x_2)
  \in \y_b\}$.  So for any $(b, x_1 ) \in \y_b$
  \begin{gather}
    \label{fibra}
    \psi\meno (b, x_1) = \{ b\} \times \Bigl ( \y_b \cap ( \{x_1\}
    \times X ) \Bigr).
  \end{gather}
  Thus \ref{uno} follows directly from Lemma \ref{interz}.  Next set
  \begin{gather*}
    \Sigma_1 (\psi):= \{ (b, x_1, x_2 ) \in |G_\aus|: \dim_{(b, x_1,
      x_2)} \psi\meno (b,x_1) \geq 1\},
  \end{gather*}
  Since $\psi$ is a proper holomorphic map between reduced complex
  spaces \cite[Thm. II.4.5.3 p. 179]{barlet-magnusson-vol-1} ensures
  that $\Sigma_1(\psi)$ is an analytic subset of $|G_\aus|$. Since
  $\psi$ is proper, its image $Z : = \psi (\Sigma_1(\psi))$ is also an
  analytic set by Remmert Proper Mapping theorem.
  Its
  complement $\Omega' : = \aus \times X -Z$ is Zariski open and it
  contains exactly the points of $\aus\times X$ whose fibre (for
  $\psi$) is 0-dimensional.  Using \eqref{fibra} and Lemma
  \ref{interz} we conclude that $\Omega' = \Omega$.  This proves
  \ref{due}.  The restriction $\psi\restr{\psi\meno (\Omega)}$ is by
  construction a continuous bijection of $\psi\meno(\Omega)$ onto
  $\Omega$.  Since the domain is locally compact and the target is
  Hausdorff, it is a closed map. This proves \ref{trebis}.  Let
  $(b,x) \in \Omega$ and assume
  $\y_b \cap (\{x\}\times X) = \{(x, x')\}$.  It follows from Lemma
  \ref{interz} that $\y_b$ is smooth at $(x,x')$ and transverse to
  $\{x\}\times X$.  Hence there is a neighbourhood $V$ of $ (x,x')$ in
  $\y_b$ such $\pi_1\restr{V}$ is a biholomorphism onto a
  neighbourhood $U\subset X$ of $x$. Set
  $\phi:= \pi_2 \circ (\pi_1\restr{V})\meno : U \ra X$. Then
  $V=\Ga_\phi$.  But $U\subset \Omega$, so
  $V = \y_b \cap (U\times X)$.  This proves \ref{tre}.  \ref{quaccio}
  is obvious.

  If $b \in \au$, then $Y_b$ has a unique component of multiplicity
  1. Therefore the definition \eqref{def-grafico} of $G_\aus$ and
  Theorem \ref {molteplicita-componenti} imply that $G_\aus$ has a
  unique component of multiplicity 1, i.e.  \ref{quinque} holds.

  If $b\in \au$ we have $\{b\} \cap X \subset \Omega$. Assume
  $b \in \B$. By \ref {uno} $\pi_1 \restr{|\y_b|} : \y_b \ra X$ is
  onto. If every fibre had positive dimension, Theorem \ref{mappotta}
  would imply that $\dim| \y_b| \geq \dim X +1$, which is absurd. So
  the fibre over some $x \in X$ has dimension 0. By Lemma \ref{interz}
  $(b,x) \in \Omega$. This proves \ref{sei}.

  Let $Y_b = \sum_{i=1}^r n_i Z_i$ be the decomposition in irreducible
  components.  Since $ \cup_i \pi_1(Z_i) =\pi_1(|\y_b|) =X$, there is
  at least one index $i$, such that $\pi_1(Z_i) = X$.  Set
  $T:= \{x\in X: (b,x) \not \in \Omega\}$.  By \ref{due} $T$ is an
  analytic subset of $X$ and by \ref{sei} it is a proper subset.  If
  $x\in X-T$, then there is exactly one $y\in X$ such that
  $(x,y) \in \y_b$.  Necessarily $(x,y) \in Z_i$ and
  $x \not \in \pi_1 (Z_j)$ for $j\neq i$.  This shows that the
  component $Z_i$ is unique and also that $\pi_1 (Z_j) \subsetneqq X$
  for $j\neq i$.  Denote by $Z_b$ the component $Z_i$.  By Theorem
  \ref{mappotta} applied to $p:=\pi_1\restr{Z_b}: Z_b \ra X$ there are
  Zariski open subsets $Z^0 \subset Z_b$ and $X^0 \subset X$, such
  that both $Z^0$ and $X^0$ are smooth and $p: Z^0 \ra X^0$ is a local
  biholomorphism.  We can assume that $X^0 \subset X- T$. So
  $p \restr{Z^0}$ is injective, hence a biholomorphism.  It follows
  that $p : Z_b \ra X$ is a modification with center $T$, hence
  $f_b: = \pi_2 \circ p\meno : X \mero X$ is a meromorphic map and the
  graph of $f_b$ coincides with $Z_b$ by Lemma \ref {lemma-1}.
\end{proof}

\begin{remark}
  In general the map in \ref{trebis} is not necessarily a
  biholomorphism.  The point is that a bijective holomorphic is
  automatically biholomorphic only if the target is weakly normal, see
  e.g.  \cite[p. 310-11 and p. 358]{barlet-magnusson-vol-1}. So one
  can only assert that $\psi\restr{\psi\meno(\Omega)}$ is a
  biholomorphism on the weak normalization of $\aus$. This kind of
  problem is quite common in the study of cycle spaces.  Indeed the
  weak normalization goes back to \cite{andreotti-norguet-67}.

\end{remark}

For $b\in \aus$ we will denote by $\z_b$ be the unique irreducible
component of $\y_b$ such that
\begin{gather*}
  \pi_1( \z_b) = X.
\end{gather*}
We will call $Z_b$ the \emph{meromorphic component} of $\y_b$.  We
will denote by $f_b$ the meromorphic map such that $\Ga_{f_b} = \z_b$.
We have $b\in \au$ iff $f_b \in \ai$.  We also denote by $A_b$ the set
of points $x\in X$ such that $(\{x\}\times X) \cap \y_b$ contains more
than one point.  This means that
\begin{gather}
  \label{eq:6}
  \{b\} \times (X- A_b) = \Omega \cap (\{b\}\times X).
\end{gather}
In other words, if $\y_b = Z_b + \sum_{i=1}^r n_i Z_i$, then
\begin{gather*}
  A_b:=\inde(f_b) \cup \bigcup_{i=1}^r \pi_1(Z_i).
\end{gather*}
The intersection $\y_b \cap ((X-A_b)\times X) $ is the graph of the
holomorphic map $f_b\restr{X-A_b}$. Let $\Mero(X)$ denote the set of
meromorphic self-maps of $X$. We have constructed a map
\begin{gather}
  \label{mappona}
  \Phi:\aus \ra \Mero(X), \quad \Phi(b):= f_b.
\end{gather}

\begin{remark}
  In general the map $\Phi$ is not injective: different points
  $b, b'\in \de B:=\aus - \au$ can have the same meromorphic
  components, i.e. $Z_b=Z_{b'}$.  The fibres of the map
  \eqref{mappona} can be even of positive dimention.  We describe such
  an example for $X=\PP^n$ based on the results of Brion
  \cite[p. 621-622]{brion-completions}.
  Set $V=\C^{n+1}$ and $X=\PP^n =\PP(V)$. Fix a basis
  $\{v_1, \lds, v_{n+1} \}$ of $V$.  Let $J =\{j_1 < \cds < j_r\}$ be
  a subset of $\{1, \lds, n\}$. Define
  \begin{gather*}
    \begin{aligned}
      &  V_0:=\spam (v_1, \lds, v_{j_1 }),\\
      &  V_i := \spam (v_{j_1+1} , \lds, v_{j_{i+1}}),  \text { for } 1 \leq i < r,\\
      & V_r :=\spam (v_{j_r+1} , \lds, v_{n+1} ),
    \end{aligned}
    \\
    V_{< k} :=\underset{i<k}{\oplus} V_i, \quad
    V_{> k} :=\underset{i>k}{\oplus} V_i, \quad \text{for } k=0, \lds,
    r,
    \\
    \tilde{Z}_i = \{(x,y,\ell) \in \PP^n \times \PP^n \times \PP(V_i):
    x \in \PP(V_{<i} + \ell), y\in \PP(V_{>i} + \ell)\}, \\ \text{for
    } i=0, \lds, r.
  \end{gather*}
  Denote by
  $\pi_{12} : \PP^n \times \PP^n \times \PP(V_i) \lra \PP^n \times
  \PP^n$ the projection.  Then the map \begin{gather*} \rho_i : =
    \pi_{12} \restr{\tilde{Z}_i} : \tilde{Z}_i \lra Z_i:=
    \pi_{12}(\tilde{Z}_i)
  \end{gather*}
  is a modification.  Set
  \begin{gather*}
    \Ga_J:=\sum_{i=0}^r Z_i \in C_n (X\times X).
  \end{gather*}
  We have $\pi_1(Z_i) = X$ iff $i=r$ and $\pi_2(Z_i) = X$ iff $i=0$.
  Thus the meromorphic component of $\Ga_J$ is $Z_r$.
  Since \begin{gather*} \tilde{Z}_r = \{ (x,y,y) \in \PP^n \times
    \PP(V_r) \times \PP(V_r) :
    x \in \PP(V_{<r} + y )\},\\
    Z_r=\{(x,y) \in \PP^n \times \PP(V_r): x \in \PP(V_{<r} + y )\},
  \end{gather*}
  the meromorphic component $Z_r$ only depends on $V_{<r}$ and $V_r$
  and there are infinitely many cycles $b \in \aus$
  sharing the same connected component.
\end{remark}

\begin{remark}
  The fibres of the map $\Phi$ in \eqref{mappona} give an equivalence
  relation $\sim$ on $\aus$ and it would be nice to prove that the
  quotient of $\aus$ with respect to this equivalence relation has the
  structure of complex analytic space.  This is indeed the case when
  $X=\PP^n$.  In fact, as shown above, the meromorphic component of a
  cycle $\Ga_J$ depends only on $V_{<r}$ e $V_r$. Moreover $\Ga_J$
  coincides with the graph of the projection onto $\PP(V_r)$ with
  centre $\PP(V_{r})$.  To get the whole of $\aus$ we let
  $ \Gl(n+1,\C)$ act on the left and on the right on the various
  cycles $\Ga_J$. In this way we get the graphs of all the elements of
  $\PP(M_{n+1}(\C))$. Thus in this case
  $\aus/\sim\, = \PP(M_{n+1}(\C))$. Unfortunately dealing with the
  general case seems rather delicate.  The fibres of $\Phi$ can be of
  different dimensions, by the previous remark.  So \cite[Satz
  1(b)]{kaup} shows that in general the relation $\sim$ is not
  open. Therefore to prove that $\aus/\sim$ is a complex space one
  cannot apply directly the main theorem of \cite{kaup}, which says
  that the quotient of a seminormal complex space by an open analytic
  relation is a complex space.
\end{remark}

  \begin{remark}
    In a series of papers Neretin gave a new construction of
    compactifications of reductive groups and symmetric spaces. In
    particular he gave a compactification of $\PGL(n+1,\C)$ via
    so-called \emph{hinges}, see \cite{neretin,neretin2}. This
    compactfication is a semigroup and it coincides with the De
    Concini-Procesi compactification \cite{deconcini-procesi}. By
    Brion's results \cite{brion-completions} it also coincides with
    $\aus$ for $X=\PP^n$. It would be very interesting to see if also
    for a general $X$ the space $\aus$ or some compactification
    related to it is a semigroup.  This would be related to the
    philosophy put forward at pages 1 and 9-11 of
    \cite{neretin-libro}. We hope to come back to these questions in
    the future.
  \end{remark}

  Consider now the following action of $\ai $ on $X\times X$:
  \begin{equation*}
    g \cd (x, y) :=  ( x, g\cd y).
  \end{equation*}
  This action induces a corresponding action on $C_n (X\times X)$: for
  $\Ga \in C_n(X\times X)$ set
  \begin{equation}
    \label{azione-suC}
    g \cd \Ga : = ( \id_X \times g) _* \Ga.
  \end{equation}
  This action preserves $\aus$.

  \begin{teo}
    \label{nondominant}
    For $b \in \B$ the stabilizer $\ai _b$ for the action
    \eqref{azione-suC} has positive dimension.  Moreover
    $f_b (X) = \pi_2(Z_{b}) \subset X^{\ai_b}$.  In particular
    $f_b:X \dasharrow X$ is non-dominant.
  \end{teo}
  \begin{proof}
    The map $j$ of \eqref{def-j} is equivariant with respect to the
    action of $\ai $ on itself by left multiplication and the action
    \eqref{azione-suC} on $C_n(X\times X)$:
    \begin{gather*}
      j (gh) = \Ga_{gh} = g \cd \Ga_h = g \cd j(h).
    \end{gather*}
    Thus $\au=j(\ai )$ is an orbit of $\ai $.  We know from
    Proposition \ref{lieber} that $\aus $ is irreducible and that
    $\B:= \aus - \au $ is a proper analytic subset of $\aus$.  Hence
    any irreducible component of $\B$ has dimension strictly less than
    $\dim\aus $.  Since $\B$ is invariant by the action, it follows
    that for $b\in\B$, $\dim \ai \cd b < \dim \au =\dim \ai$, so
    $\dim \ai _b >0$.

    Denote by $\y_b$ the cycle corresponding to $b$ and let $Z_b$ be
    the meromorphic component. If $g \in \ai$, then clearly
    $\pi_1 (g \cd \Ga_b) = \pi_1 (\Ga_b) $. Thus for $g \in \ai_b$,
    $g\cd Z_b = Z_b$.

    Let $Y:=\inde(f_b) \subset X$ be the indeterminacy locus of
    $f_b: X \mero X$.  If $x \in X \setminus Y$, then
    $(\{x\}\times X) \cap Z_b = \{(x, f_b(x))\}$.  If $h\in \ai_b$,
    then $h \cd (x, f(x)) = (x, hf(x)) \in Z_b$, so $hf(x) = f(x)$.
    This shows that $f_b(X \setminus Y) \subset X^{\ai_b}$.  Since
    $Z_b$ is the closure of $\{(x, y) \in (X-Y)\times X: y=f_b(x)\}$,
    we conclude that $\pi_2 (Z_b) \subset X^{\ai_b}$.

    Finally, since $\ai_b$ has positive dimension, it is not the
    trivial subgroup, so $X^{\ai_b}$ is a proper analytic subset of
    $X$. Therefore the image of $f_b$ is strictly smaller than $X$.
  \end{proof}

\begin{remark}
  A refinement of this theorem in the case of a reductive subgroup is
  given by Theorem \ref{nondominantG} below.
\end{remark}

Theorem \ref{main1} in the Introduction follows from the previous
theorem together with the following one.

\begin{teo}
  \label{convergenza}
  Let $X $ be a compact complex manifold in the class $\fuj$.  Let
  $\{b_j \}$ be a sequence in $ \aus$ converging to $ b \in \aus$.
  Then $f_{b_j} \to f_b$ uniformly on compact subsets of $X-A_b$.  In
  particular, if $\{g_j\}$ is a sequence in $\ai$, passing to a
  subsequence we can find $b\in \aus$ such that $g_j \to f_b$
  uniformly on compact subsets of $X-A_b$.
\end{teo}

\begin{remark}
  In general the set $A_b$ is larger than the indeterminacy set of
  $f_b$ and the convergence holds only on $X-A_b$.  For example if
  $X=\PP^1$ and $g_j $ is the map $g_j(z)=j\cd z$, then $f_b$ maps
  every point of $\PP^1$ to $\infty$ and has no indeterminacy point,
  but convergence does not hold at $0\in A_b$.
\end{remark}

We start the proof with the following elementary observation.

\begin{lemma}
  Let $X$ and $Y$ be topological spaces and let $(Z,d)$ be a metric
  space.  Let $h: X\times Y \ra Z$ be a continuous map.  Let $\{x_n\}$
  be a sequence in $X$ converging to $\bar{x} \in X$. Set
  \begin{gather*}
    f_n(y):= h(x_n, y), \quad \bar{f}(y):=h(\bar{x}, y).
  \end{gather*}
  If $Y$ is compact, $f_n \to \bar{f}$ uniformly on $Y$.
\end{lemma}
\begin{proof}
  Fix $\eps >0$. Given $y_0\in Y$, continuity of $h$ yields open
  neighbourhoods $U$ of $\bar{x}$ in $X$ and $V$ of $y_0$ in $Y$, such
  that $d(h(x, y), h(\bar{x}, y_0)) < \eps /2$ for any
  $(x,y) \in U\times V$.  Since $Y$ is compact we can cover it with a
  finite number of neighbourhoods like $V$, that is we can find a list
  $\{(U_i, V_i, y_i)\}_{i=1}^n$ such that $U_i$ is open in $X$, $V_i$
  is open in $Y$, $\bar{x} \in U_i$, $y_i \in V_i$, $\cup_i V_i = Y$
  and
  \begin{gather}
    \label{eq:1}
    d(h(x, y), h(\bar{x}, y_i)) < \frac{\eps}{2} , \quad \forall (x,y)
    \in U_i \times V_i.
  \end{gather}
  Then $W:=\cap _i U_i$ is a neighbourhood of $\bar{x}$, so there is
  $n_0$ such that for any $n\geq n_0$, $x_n \in W$.  If $y\in Y$,
  there is $i$ such that $y\in V_i$. Hence for $n\geq n_0$ using twice
  \eqref{eq:1} we get
  \begin{gather*}
    d (h(x_n, y), h(\bar{x}, y)) \leq d (h(x_n, y), h(\bar{x}, y_i)) +
    d (h(\bar{x}, y), h(\bar{x}, y_i)) < \eps.
  \end{gather*}
\end{proof}

\begin{proof}
  [Proof of Theorem \ref{convergenza}] Fix a compact subset
  $K \subset X-A_b$.  By \eqref {eq:6} this means that
  $\{b\} \times K \subset \Omega$, so there is an open subset
  $V \subset \aus$ such that $V \times K \subset \Omega$.  There is
  $n_0$ such that $b_j \in V$ for $n\geq n_0$.  Recall from Lemma
  \ref{cavallo} \ref{tre} that $\psi\restr{\psi\meno(\Omega)} $ is a
  homeomorphism. In particular we can invert
  $f:=\psi\restr{V\times K} $. Hence we have a well-defined map
  \begin{gather*}
    h:= \pi_3 \circ f\meno : V\times K \lra X.
  \end{gather*}
  By Lemma \ref{convergenza} $h(b_j, \cd ) \to h(b, \cd)$ uniformly on
  $X$ (with respect to any metric inducing the topology).  But if
  $n\geq n_0$, $\{b_j\} \times K \subset \Omega$, i.e.
  $K \subset X -A_{b_j}$.  Hence $h(b_j , \cd ) = f_{b_j}$ and
  $h(b , \cd ) = f_{b}$. We have proved that $f_{b_j} \to f_b$
  uniformly on $K$.
\end{proof}

\begin{remark}
  \label{ciuccia}
  It is important to notice that Theorem \ref{convergenza} does not
  hold without the hypothesis $X \in \fuj$.  Consider the following
  example already studied in \cite{ueno-class-C}.  Set
  $W:= \C^2 -\{0\}$ and choose $\alfa\in \C$ with $0<|\alfa|<1$.  Let
  $\alfa$ act on $W$ by the rule
  $\alfa\cd (x,y) := (\alfa x, \alfa y)$.  Then
  $H_\alfa : = W / \langle \alfa \rangle $ is a Hopf surface and
  $\Aut(H_\alfa) = \Gl(2,\C) / \langle \alfa \rangle$.  Set
  \begin{gather*}
    g:= \begin{pmatrix} 1 & 0 \\ 0 & \alfa
    \end{pmatrix},
  \end{gather*}
  and consider the sequence $\{g^n\} $ in $\Aut(H_\alfa)$.  Set
  $E_1 = \{ [x,y] \in H_\alfa : x=0\}$ and
  $E_2 = \{ [x,y] \in H_\alfa : y=0\}$. These are elliptic curves
  isomorphic to $\C / (\Zeta + \Zeta a)$ where
  $\exp(2\pi i a ) = \alfa$.  It is easy to check that for
  $p=[x,y] \not \in E_1$ we have
  $g^n (p) \to \phi(p) := [x,0] \in E_2$.  While for
  $p= [0,y] \in E_1$, $g^n (p) =p$.  So the limit exists for every
  $p\in H_\alfa$.  On the other hand the map
  $\phi: H_\alfa - E_1 \to E_2$ is not meromorphic.  In fact call
  $\Gamma$ its graph.  We claim that
  $\overline{\Ga} = \Ga \, \cup\, E_1\times E_2$.  It is clear that
  $\Gamma \subset H_\alfa \times E_2$ and that $\Gamma$ is closed in
  $(H-E_1) \times E_2$. Moreover if
  $([0,y], [u,0]) \in E_1 \times E_2 $, then
  \begin{gather*}
    ([\alfa^n u, y ] , \phi ([\alfa^n u, y ]) )= ([\alfa^n u, y ], [u,
    0] ) \to ([0,y], [u,0]).
  \end{gather*}
  So $E_1\times E_2 \subset \overline{\Gamma}$. This proves that
  indeed $\overline{\Ga} = \Ga \, \cup\, E_1\times E_2$.  Now we show
  that $\overline{\Gamma}$ is not analytic.  Call
  $\pi: W\times W \ra H_\alfa \times H_\alfa$ the projection.  Fix
  $p_0 =([0,y_0],[u_0,0]) \in E_1\times E_2$.  Let $U$ be a small
  neighbourhood of $(0,y_0,u_0,0)$ in $W\times W$ such that
  $\pi\restr{U}$ is a biholomorphism.  Then
  \begin{gather*}
    \pi\meno \overline{\Gamma} \cap U = (\{(0, y, u, 0)\}\cap U) \cup
    \bigcup_{n\in \Zeta} ( \{ (x, y, \alfa^n x, 0) \} \cap U),
  \end{gather*}
  which is not analytic.  One can also deduce that $\overline{\Ga}$ is
  not analytic from the fact that
  $E_1 \times E_2 \subset \overline{\Ga} - \Ga$ and
  $\dim E_1\times E_2 = \dim \Ga$.
\end{remark}

\begin{remark}
  \label{convergi}
  In the literature there are several notions of convergence for
  meromorphic maps, see for example \cite {iva,in}.  It would be
  interesting to compare the convergence in $\aus$ with these notions
  of convergence. We leave this for further inquiry.
\end{remark}

\section{Compactifications of reductive subgroups}
\label{ridut}

In this section we consider complex reductive subgroups of $\ai$.
Since we will only consider \emph{complex} reductive subgroups, we
will often refer to them simply as \emph{reductive subgroups} of
$\ai$.

Our goal is to construct compactifications of the connected reductive
subgroups of $\ai$ that act trivially on $\Alb X$.  We will take
advantage of Fujiki's deep work in \cite{fujiki-automorphisms}.  We
start by recalling some definitions introduced in that paper.

Let $G$ be a connected complex Lie group. A \emph{meromorphic
  structure} on $G$ is an analytic compactification $G^*$ (i.e. a
compact analytic space $G^*$ containing $G$ as a dense open subset)
such that the product map and the inversion extend as meromorphic maps
$G^* \times G^* \mero G^*$ and $G^* \mero G^*$.  Two such structures
$G^*$ and $G^{**}$ are \emph{equivalent} if $\id_G$ extends to a
bimeromorphic map $G^* \mero G^{**}$.  An equivalence class of
meromorphic structures is called a \emph{meromorphic group}.  We will
denote a meromorphic group by $G$ or $G^*$ or $(G,G^*)$.

If $G^*$ is a meromorphic structure on $G$, a subgroup $H\subset G$ is
\emph{meromorphic} if the closure of $H$ in $G^*$ is an analytic
subset.  If $G^{**}$ is another meromorphic structure which is
equivalent to $G^*$, then $H$ is a meromorphic subgroup with respect
to $G^*$ iff it is meromorphic with respect to $G^{**}$.  To prove the
last statement one uses Lemma \ref{chiusura}. Thus the notion of
meromorphic subgroup depends only on the ambient meromorphic group.

If $G$ is a linear algebraic group over $\C$, then it has a canonical
meromorphic structure given by taking a faithful representation of
$G \ra \Sl(V)$ and letting $G^*$ be the closure of $G$ inside
$\PP(\End V)$.  This structure is well-defined, i.e. does not depend
on the choice of the representation
\cite[Rmk. 2.3]{fujiki-automorphisms}. When $G$ is endowed with this
structure we say that it is \emph{meromorphically linear}.

If $G$ is a connected complex Lie group with a meromorphic structure
$G^*$ and $X$ is a complex space we say that an action
$\sigma : G \times X \lra X$ of $G$ on $X$ is \emph{meromorphic} if
$\sigma$ extends to a meromorphic map $G^* \times X \mero X$.

\begin{prop}
  \label{bimer}
  Let $(G,G^*)$ be a meromorphic group. Assume that $G$ acts on the
  compact complex spaces $X$ and $Y$ and that $f: X \mero Y$ is a
  $G$-equivariant bimeromorphic map. Then the action on $X$ is
  meromorphic iff the action on $Y$ is meromorphic.
\end{prop}

\begin{proof}
  Let $X^0$ and $Y^0$ be Zariski open subsets such that
  $f: X^0 \lra Y^0$ is a biholomorphism. Equivariance is understood in
  the following sense: if $x\in X^0$ and $g\cd x \in X^0$, then
  $f(g\cd x) = g\cd f(x)$.  Denote by $\sigma: G\times X \ra X$ the
  action on $X$ and by $\tau: G\times Y \ra Y $ that on $Y$.  Set
  $F:=\id_{G^*}\times f\times f : G^* \times X^2 \mero G^*\times X^2$.
  Consider the set
  $\Ga':=\{(g,x_1,x_2) \in G\times X ^2: x_1, x_2\in X^0, x_2=g\cd
  x_1\}$.  Since $\Ga'=\Ga_\sigma \cap (G\times X^0 \times X^0)$, it
  is a Zariski open subset of $\Ga_\sigma$. It is clearly non-empty
  since $(1, x, x) \in \Ga'$ for any $x\in X^0$.  Therefore it is
  dense in $\Ga_\sigma$.  The same holds for
  $\Ga''=\Ga_\tau \cap (G\times Y^0 \times Y^0)$: this is a dense
  Zariski open subset of $\Ga_\tau$.  The map $F$ is defined on
  $\Ga'$.  The equivariance and the hypothesis on $X^0$ and $Y^0$
  imply that $F(\Ga') = \Ga''$.  Denote by $W'$ the closure of
  $\Ga_\sigma$ in $G^*\times X \times X$.  If the action of $G$ on $X$
  is meromorphic, $W'$ is an analytic subset of $G^*\times X\times X$
  by Lemma \ref{lemma-1}.  Since $\Ga_\sigma$ is closed in
  $G\times X \times X$, we have
  $W' \cap ( G\times X \times X) = \Ga_\sigma \cap ( G\times X \times
  X) $.  Let $A$ be the complement of $G\times X^0 \times X^0$ in
  $G^*\times X\times X$.  $A$ is an analytic subset and it contains
  $\inde(F)$.  The set $W'$ is irreducible and it is not contained in
  $A$.  So Lemma \ref {chiusura} implies that
  $W'':= \overline{ F(W'-A)}$ is an analytic subset of
  $G^*\times Y\times Y$.  But by the definition of $\Ga'$ we have
  $\Ga'=W'-A$.  So $W'' = \overline{\Ga''}$. But we know that $\Ga''$
  is dense in $\Ga_\tau$, so
  $\Ga'' \subset \Ga_\tau \subset \overline{\Ga''} $. This finally
  shows that $\overline{\Ga}_\tau = W''$ is analytic, i.e. the action
  on $Y$ is meromorphic.
\end{proof}

Assume that $X$ is a compact complex manifold.  Let $\Fu(X)$ denote
the irreducible component of the Douady space $\Dou (X\times X)$
containing the diagonal $\Delta$. We let $\Fu(X)_\red$ denote the
reduction of $\Fu(X)$.  We recall some fundamental results of
  Fujiki.

\begin{teo}
  [Fujiki]\label{fujiki} If $X \in \fuj$, then $\Fu(X)_\red$ is a
  meromorphic structure on $\ai$, called the \emph{natural}
  meromorphic structure.  Moreover there is an exact sequence of
  meromorphic groups
  \begin{gather*}
    0 \to L(X) \to \ai \stackrel{\alfa}{\lra} T(X)\to 0
  \end{gather*}
  where $L(X)$ is meromorphically linear and $T(X)$ is a torus.
\end{teo}
See \cite[Prop. 2.2 p. 231]{fujiki-automorphisms} and
\cite[Thm. 5.5]{fujiki-automorphisms}.  If $H\subset \ai$, we say that
$H$ is a \emph{meromorphic subgroup with the natural structure} if it
is a meromorphic subgroup of $\Fu(X)_\red$, i.e. if $\bar{H}$ is an
analytic subset of $\Fu(X)_\red$.

Let $\Alb X $ be the Albanese torus of $X$.  Since $\Alb X$ is a
compact torus, the group $A(X):=\Aut^0(\Alb X)$ is simply the group of
translations of $\Alb X$.  If $x_0\in X$ is fixed, one defines an
Albanese map $\alb_X: X \ra \Alb X$ with $\alb(x_0) = 0$ and a
homomorphism
\begin{gather*}
  \Aut(X) \lra \Aut(\Alb X), \quad g \mapsto A_g
\end{gather*}
such that $\alb_X \circ g = A_g \circ \alb_X$ for every
$g \in \Aut(X)$ \cite[p. 101]{akhiezer-libro}.  The \emph{Jacobi
  morphism} $\phi_*: \ai \lra A(X)$ is defined as the restriction of
the morphism $g\mapsto A_g$ to the connected components of the
identity.

\begin{prop} [\protect{\cite[Thm. 5.5 (2) p.
    251]{fujiki-automorphisms}}] If $X\in \fuj$, then $L(X)$ is a
  finite index subgroup of $\ker \phi_*$.
\end{prop}

\begin{cor}\label{pinco}
  If $X \in \fuj$ and $G\subset \Aut^0(X)$ is a connected subgroup,
  then $G$ acts trivially on $\Alb X$ if and only if $G\subset L(X)$.
\end{cor}

\begin{cor}\label{Kahmero}
  If $X$ is K\"ahler and $K$ is a compact connected Lie group that
  acts holomorphically on $X$ in Hamiltonian way, then $G:=K^\C$ is a
  meromorphic subgroup of $\ai$.
\end{cor}
\begin{proof}
  The assumption means that there are a K\"ahler form $\om$ and a
  momentum mapping $\mu: X\lra \liek^*$ such that $\om$ is
  $K$-invariant, $\mu$ is equivariant and
  $d\sx \mu, v\xs = i_{\xi_v}\om$, where $\scalo$ denotes the pairing
  of $\liek^*$ and $\liek$ and $\xi_v$ is the fundamental vector field
  corresponding to $v\in \liek$.  It is well-known that $K$ acts by
  biholomorphisms \cite[p. 93]{kobayashi-trans}, that the inclusion
  $K \subset \ai$ extends to an inclusion $G:=K^\C \subset \ai$ and
  that $G$ acts trivially on $\Alb X$, \cite[Prop. 1]{alan-wurzy}.
\end{proof}

\begin{teo}[Fujiki]
  \label{pallino}
  Let $X \in \fuj$ and let $G \subset \ai$ be a connected reductive
  subgroup.  Then $G$ is meromorphic (with the natural structure) if
  and only if it acts trivially on $\Alb X$.
\end{teo}
\begin{proof}
  One implication is proved in \cite[Lemma 3.8]{fujiki-automorphisms}.
  For the other assume that $G$ acts trivially on $\Alb X$.  By
  Corollary \ref{pinco} $G\subset L(X)$.  By Theorem \ref{fujiki}
  $L(X)$ is a meromorphic subgroup of $\Fu(X)_\red$ and the
  meromorphic structure induced from $\Fu(X)_\red$ (i.e. the natural
  structure) is equivalent to the linear one.  Since $G$ is reductive,
  it is an algebraic subgroup of $L(X)$.  Hence it is a meromorphic
  subgroup of $L(X)$ with the natural structure and thus it is itself
  a meromorphic subgroup of $\Autz(X)$ with the natural structure.
  See \cite[Prop. 6.10]{fujiki-automorphisms}.
\end{proof}

\begin{prop}
  \label{fubamero}
  If $X$ is a compact complex manifold, then $\Fu(X)_\red$ is
  $\ai$-equivariantly bimeromorphic to $\aus$.
\end{prop}
\begin{proof}
  The morphism from Douady space to cycle space restricts to a
  surjective holomorphic map $f : \Fu(X)_\red \lra \aus$, see
  \cite[Thm. 8 p. 121]{barlet-482}.  This map is obviously
  $\ai$-equivariant. The complex space $\ai$ embeds in both
  $\Fu(X)_\red$ and $\aus$. If we consider these embeddings as
  identifications, the map $f$ extends $\id_{\ai}$. In particular $f$
  is 1-1 over $\ai$. By Lemma \ref{bimeromorfi} $f$ is bimeromorphic.
\end{proof}

\begin{prop}
  \label{fuba}
  If $G$ is a meromorphic subgroup with the natural structure, then the
  closure of $G$ in $\aus$ is an analytic subset.
\end{prop}
\begin{proof}
  Consider again the morphism from Douady space to cycle space
  $f : \Fu(X)_\red \lra \aus$ as in Proposition \ref{fubamero}.
  Denote by $W$ the closure of $G$ in $\Fu(X)_\red$. By assumption $W$
  is an analytic subset.  By Remmert Proper Mapping Theorem $f(W)$ is
  an analytic subset of $\aus $. But it coincides with the closure of
  $G$ in $\aus$.
\end{proof}

\begin{cor} If $X$ is K\"ahler and $K$ is a compact connected Lie
  group that acts holomorphically on $X$ in Hamiltonian way, then
  $G:=K^\C$ has analytic closure in $\aus$.
\end{cor}
\begin{proof}
  By Corollary \ref{Kahmero} $G$ is meromorphic.
\end{proof}

The next result is a refinement of Theorem \ref{nondominant}.

\begin{teo}
  \label{nondominantG}
  Assume that $X \in \fuj$ and that $G\subset \ai$ is a meromorphic
  subgroup (in the natural structure).  Let $\bar{G} $ denote the
  closure of $G$ in $\aus$ and set $ \partial G:=\bar{G} -G$.  Then
  $ \partial G \subset \partial \aus$. Morever for $b \in \partial G$,
  the stabilizer $G _b$ for the action \eqref{azione-suC} has positive
  dimension and $f_b(X) \subset X^{G_b}$.
\end{teo}
\begin{proof}
  Let $j$ be the map defined in \eqref{def-j} and consider the action
  of $\ai$ on $C_n(X\times X)$ defined in \eqref{azione-suC}.  As
  usual we identify elements of $\ai $ with their image through $j$.
  So we consider $G\subset \ai=\au \subset \aus$.  By Proposition
  \ref{fuba} $\bar{G}$ is an analytic subset of $\aus$.  In particular
  $G$ is closed in $\au$, so $ \partial G \subset\partial \aus$.  To
  prove the second assertion, observe that $G $ is an open orbit of
  itself in $\bar{G}$.  By Proposition \ref{lieber} $\bar{G} $ is
  irreducible and $\partial G $ is a proper analytic subset of
  $\bar{G}$.  Hence any irreducible component of $\partial G$ has
  dimension strictly less than $\dim G $.  Since $\partial G$ is
  invariant by the action, it follows that for $b\in\partial G$,
  $\dim G \cd b < \dim G$, so $\dim G _b >0$.  Observing that
  $G_b \subset \ai_b$ and applying Theorem \ref{nondominant} concludes
  the proof.
\end{proof}

If $X$ is K\"ahler, we can say something on the geometry of $\bar{G}$.
(Compare Theorem \ref{main2}.)
\begin{teo}
  \label{proj}
  If $X$ is a K\"ahler manifold and $G \subset \ai$ is a connected
  reductive subgroup, that acts trivially on $\Alb X$, then the
  closure of ${G}$ inside $\aus$ is a projective variety.
\end{teo}
\begin{proof}
  By Corollary \ref{pinco} $G \subset L(X)$. By Theorem \ref{pallino}
  $\bar{G}$ is an analytic subset of $\aus$. By a result of Varouchas
  \cite{varouchas-image} $\aus$ is a K\"ahler space, so the same is
  true of $\bar{G}$. Let $\pi: Z\ra \bar{G}$ be a $G$-resolution of
  $\bar{G}$ (see e.g. \cite[p. 150]{kollar-resolution}).  Then $\pi$
  is a projective, hence a K\"ahler morphism \cite[Prop. 4.6
  (4)]{bing}.  Since $Z$ is compact, it follows from \cite[Prop. 4.6
  (2)]{bing} that it is K\"ahler. Thus $Z$ is a K\"ahler $G$-almost
  homogeneous manifold. We claim that $G$ acts trivially on $\Alb(Z)$.
  Indeed $G$ acts on $\Alb Z$ and being connected it acts by
  translations.  Now up to a finite cover $G = T \rtimes S$ with
  $T=(\C^*)^r$ and $S$ semisimple and connected.  Any morphism
  $S \ra \Alb Z$ is trivial, so $S$ acts trivially.  Each
  $\C^*$-factor of $T$ is algebraic in $G$ and hence is a meromorphic
  subgroup of $G$. As such $\C^*$ acts meromorphically on $X$. By
  \cite[Prop. 2.2]{fujiki-automorphisms} it acts meromorphically also
  on $\Fu(X)$ and on $\Fu(X)_\red$.  Using Propositions \ref{bimer}
  and \ref{fubamero} we conclude that the action of $\C^*$ on $Z$ is
  meromorphic.  Hence every orbit has analytic closure \cite[Lemma 2.4
  (1)]{fujiki-automorphisms}.  Fix $z \in Z$. The closure of
  $\C^*\cd z$ contains a closed orbit, i.e. a fixed point.  So fixed
  points exists, hence $\C^*$ acts trivially on $\Alb Z$
  \cite{sommese-vector}.  By \cite[Prop. 2]{alan-wurzy} and
  \cite{oeljeklaus} we get that $b_1(Z) = 0$ and $Z$ is projective.
  It follows that $\bar{G}$ is Moishezon, since it is bimeromorphic to
  the projective manifold $Z$, see \cite[p. 305]{petermod}. But
  $\bar{G}$ is also K\"ahler. Being Moishezon and K\"ahler $\bar{G}$
  is in fact projective by \cite[p. 310]{petermod}.
  
\end{proof}

\begin{remark}
  It would be interesting to know if $\aus$ is projective for any
  $X\in \fuj$, without the K\"ahlerness assumption.
\end{remark}

\section{The action on the set of measures}
\label{misure}

If $X$ is a compact manifold, denote by $\misu(X)$ the vector space of
finite signed Borel measures on $X$ endowed with the weak topology.
Denote by $\proba(X) \subset \misu(X)$ the set of Borel probability
measures on $X$.

The following theorem is a generalization of the so-called Furstenberg
lemma, which corresponds to the case $X=\PP^n$, see
\cite{furstenberg-Borel}, \cite[IV]{zimmer-ENS},\cite[Lemma
3.2.1]{zimmer-book}

\begin{teo}\label{fursty-analitico}
  Let $X$ be a complex manifold in the class $\fuj$.  Let
  $\mu, \nu\in \proba(X)$ and let $\{g_n\} $ be a sequence in $\ai$,
  such that $g_n\cd \mu \weak \nu$.  Then either $\{g_n\}$ has compact
  closure in $\ai$ or $\nu$ is supported on a proper analytic subset
  of $X$.
\end{teo}
\begin{proof}
  If $\{g_n\}$ is divergent in $\ai$, we can extract a subsequence
  (that we still denote by $\{g_n\}$) converging to some $b\in \B$.
  By Theorems \ref{convergenza} and \ref{nondominant} we have
  \begin{enumerate}
  \item [a)] $g_n \to f_b$ uniformly on compact subset of $X -A_b$;
  \item [c)] $f_b (X-A_b) \subset A':=X^{\ai_b} \subsetneq X$.
  \end{enumerate}
  Let $A_j$ be the irreducible components of $A_b$ and set
  $a_j:=\dim A_j$. For any fixed $j$ the cycles $g_n \cd A_j$ belong -
  for any $n$ - to the same irreducible component of
  $C_{a_j}(X)$. These components are compact by Theorem \ref{C}, so by
  passing to a subsequence we can assume that $g_n \cd A_j \to \A_j$
  for any $j$ and for some $\A_j \in C_{a_j}(X)$. The convergence as
  cycles implies the analogous convergence as closed subset of the
  metric space $X$. \cite[Cor. 2.7.13
  p. 424]{barlet-magnusson-vol-1}. Hence, writing $\A:=\cup_j \A_j$,
  we have
  \begin{enumerate}
  \item [c)] $g_n\cd A \to \A$ in the Hausdorff topology of closed
    subsets.
  \end{enumerate}
  Write $\mu= \mu_1 + \mu_2$ with $\mu_1 (X-A) = \mu_2(A) = 0$. Since
  $\proba(X)$ is compact in the weak topology, up to passing to a
  subsequence we can assume that $g_n\cd \mu_1 \weak \nu_1$ and
  $g_n\cd \mu_2 \weak \nu_2$. Hence $\nu_1 + \nu_2 = \nu$.  We claim
  that
  \begin{enumerate}
  \item [d)] $\supp(\nu_1) \subset \A$;
  \item [e)] $\supp(\nu_2) \subset A'$.
  \end{enumerate}
  To prove (d) fix $u \in C(X)$ such that $\supp(u) \cap \A = \vacuo$.
  Then there is $\eps>0$ such that $\supp(u) \cap (\A)_\eps =
  \vacuo$. So $\supp(u) \cap (g_n \cd A) =\vacuo$ for large $n$.
  Now
  \begin{gather*}
    \int_X u\, d\nu_1 = \lim_{n\to \infty} \int _X u\, d (g_n \mu_1) =
    \lim_{n\to \infty} \int _X u(g_n \cd x)\, d \mu_1(x) ,
    \\
\text{and} \quad    \int _X u(g_n \cd x)\, d \mu_1(x) = \int _A u(g_n \cd x)\, d
    \mu_1(x),
  \end{gather*}
  since $\mu_1$ is concentrated on $A$.  For large $n$ the last
  integral vanishes, since $u$ vanishes on $g_n\cd A$.  This proves
  (d).

  To prove (e) fix $u\in C(X)$ with $\supp(u)\cap A' = \vacuo$.  As
  before
  \begin{gather*}
    \int_X u\, d\nu_2 = \lim_{n\to \infty} \int _X u(g_n \cd x)\, d
    \mu_2(x) .
  \end{gather*}
  By (a) we have $u(g_n\cd x) \to u (f_b(x))$ pointwise on $X-A$,
  hence $\mu_2$-a.e.  Since $u\in L^\infty$ we can apply Lebesgue
  Dominated Convergence Theorem to get
  \begin{gather*}
    \int_X u\, d\nu_2 =\int_X u(f_b(x)) \, d\mu_2 (x).
  \end{gather*}
  But $f_b(X-A) \subset A'$ by (b). Since $u\equiv 0 $ on $A'$, we
  conclude that $\int_X u\, d\nu_2 = 0$.  So (e) also is proven.  (d)
  and (e) together clearly imply that $\supp(\nu) \subset \A \cup A'$,
  so the theorem is proved.
\end{proof}

The following was already known in the special case $X=\PP^n$, see
\cite[Cor. 3.2.2, p. 39]{zimmer-book}.

\begin{cor}\label{corzimmer}
  If $\nu \in \pb$, then
  \begin{itemize}
  \item [i)] either $\nu$ is not supported on a proper analytic
    subset, in which case $\ai_{\nu}$ is compact;
  \item [ii)] or there is a proper irreducible analytic subset $Y$ of
    $X$ such that
    \begin{enumerate}
    \item [a)] $\nu(Y) >0$,
    \item [b)] the orbit
      $ O:= \ai_\nu \cd Y = \{ g\cd Y | g\in \ai_\nu\} $ is finite; in
      particular a finite subgroup of $\ai_\nu$ leaves $Y$ invariant.
    \end{enumerate}
  \end{itemize}
\end{cor}
\begin{proof}
  If $\nu$ is not supported on a proper analytic subset, the previous
  theorem implies that $\ai_\nu$ is compact. Thus (i) is clear.  If
  $\nu$ is supported on a proper analytic subset, then there are
  proper irreducible analytic subsets with $\nu(Y) >0$. Take $Y$ to be
  one of minimal dimension. If $g_1 \cd Y $ and $g_2 \cd Y$ are
  distinct elements of $O$, then $\nu (g_1 \cd Y \cap g_2 \cd Y)
  =0$. Otherwise some irreducible component $Z$ of this intersection
  has positive measure and $\dim Z < \dim Y$. Moreover since
  $g_ i\in \ai_\nu$ we have $\nu (g_1 \cd Y ) = \nu (g_2 \cd
  Y)$. Since $\nu(X) = 1$ the orbit must be finite. The rest is clear.
\end{proof}

\begin{remark}
  We remark that in fact one might expect a better result: linear
  subspaces of $\PP^n$ can be characterized as fixed sets of subgroups
  of $\operatorname{PGL}(n+1,\C)$. So one might ask if the support of
  a measure with non-compact stabilizer is in fact contained in the
  fixed set of a proper subgroup of $\ai$. We leave this point for
  further inquiry.
\end{remark}

  Another application concernes the construction of Hersch and
  Bour\-gui\-gnon-Li-Yau that we now recall briefly, see \cite[\S\S
  5-6]{bgs} for more details.  Let $X$ be a compact K\"ahler manifold
  and let $K$ be a compact connected Lie group acting almost
  effectively on $X$ with momentum mapping
  \begin{gather*}
    \mu : M \ra \liek^*.
  \end{gather*}
  If $v\in \liek$, set $\mu^v :=\sx \mu, v\xs $.  Then $\mu$ is
  $K$-equivariant and $d\mu^v = i_{v_M}\om$.  The action of $K$
  extends to a holomorphic action of the complexification $G:=K^\C$.
  Define $\fun : \pb \ra \liek^*$ by the formula
  \begin{gather*}
    \fun (\nu) : = \int_M \mu(x) d\nu(x).
  \end{gather*}
  As explained in \cite{bgs} this map is a momentum mapping for the
  action of $K$ on $\pb$, in an appropriate sense.

  Let $\convo$ denote the convex hull of $\mu(M) \subset \liek^*$ and
  let $\intec $ denote the interior of $\convo$ as a subset of
  $\liek^*$.  Finally set
  \begin{gather*}
    \bly_\nu : G \lra \liek^*, \quad \bly_\nu(a): = \fun (a\cd \nu).
  \end{gather*}
  The following  should be compared to Theorem 6.14 in \cite{bgs}.

\begin{teo}\label{propria}
  Fix $\nu \in \pb$ and assume that $\nu(A) = 0$ for any proper
  analytic subset $A$ of $X$.  Then $F_\nu(G) = \into$ and
  $F_\nu : G\ra \into$ is a fibration with compact connected fibres.
\end{teo}
\begin{proof}
  By Corollary \ref {corzimmer} (i) the stabilizer $\ai_\nu$ is
  compact, so also $G_\nu$ is compact.  Therefore Theorem 6.4 in
  \cite{bgs} implies that the map $F_\nu$ is a smooth submersion onto
  its image, which is an open subset of $\intec$.  To conclude it is
  enough to check that $F_\nu$ is proper as a map $G \ra \intec$ (see
  \cite[p. 1140]{bgs} for details).  Let $\{g_n\}$ be a diverging
  sequence in $G$. Since $\convo$ is compact, we can assume that
  $F_\nu(g_n) \to \xi \in \convo$.  We have to prove that
  $\xi \in \partial \convo$.  If $\bar{G}$ denotes the closure of $G$
  in $\aus$ (which is compact), we can also assume $g_n \to b $ for
  some $b\in \partial G$. Let $\nuo$ be a fixed smooth probability
  measure, i.e. a measure given by a smooth volume form which vanishes
  nowhere.  By Theorem 6.14 of \cite{bgs} (see also Definition 5.27 in
  that paper) the map $\bly_{\nuo} : G \lra \Omega(\mu)$,
  $ \bly_\nuo(a): = \fun (a\cd \nuo)$ is proper.  Therefore up to
  passing to a subsequence we can assume that $\bly_{\nuo } (g_n) $
  converges to some point $ \theta \in \partial E (\mu)$.  And by
  Theorem 0.3 in \cite{bgh-israel-p}, the convex body $E(\mu)$ has the
  property that all its faces $\overline{\Omega(\mu)}=\convo$ are
  exposed (see \cite[p. 426]{bgh-israel-p} for the
  definitions). Therefore there exists a $v\in \liek$, such that
  $v\neq 0$ and
  $\langle \theta,v \rangle= \max_{\convo} \sx \cd , v \xs$.  On the
  other hand Theorem \ref{convergenza} we have pointwise convergence
  $g_n \to f_b(x)$ on $X-A_b$. Since $\nuo (A_b)=0$ and
  $\mu^v := \sx \cd, v\xs$ is bounded, the dominated convergence
  theorem yields
  \begin{gather*}
    \sx \bly_{\nuo} (g_n), v\xs = \int_{X} \mu^v(g_n \cd x) d\nuo (x)
    \to \int_{X} \mu^v(f_b( x)) d\nuo (x).
  \end{gather*}
  On the other hand
  $\sx \bly_{\nuo} (g_n), v\xs \to \sx \theta , v\xs$.  Thus
  \begin{gather*}
    \int_{X} \mu^v(f_b( x)) d\nuo (x) = \sx \theta, v\xs =
    \max_{\convo} \sx \cd , v \xs = \max_X \mu^v.
  \end{gather*}
  This shows that the equality $\mu^v \circ f_b = \max _X\mu^v$ holds
  $\nuo$-almost everywhere on $X-A_b$. Since this function is
  continuous, the equality holds in fact everywhere on $X-A_b$.
  But since $\nu(A_b) =0$ by assumption, we can redo this computation
  with $\nu$ instead of $\nuo$:
  \begin{gather*}
    \sx \bly_{\nu} (g_n), v\xs q = \int_{X-A_b} \mu^v(g_n \cd x) d\nu
    (x) \to \int_{X-A_b} \mu^v(f_b( x)) = \max_X\mu^v.
  \end{gather*}
  Summing up we get
  $\sx \xi , v \xs = \max_\mu^v = \max _{\convo} \sx \cd , v
  \xs$. Therefore $\xi$ (just as $\theta$) lies in the face
  $F_v (\convo)$. In particular $\xi \in \partial \convo$.
\end{proof}

\def\cprime{$'$}

\end{document}